\documentclass[11pt]{amsart}
\usepackage[top=4.2cm, bottom=4cm, left=2.4cm, right=2.4cm]{geometry}
\usepackage[utf8]{inputenc}
\usepackage[USenglish]{babel}
\usepackage[T1]{fontenc} 
\usepackage{mathrsfs}
\usepackage{mathtools}
\usepackage{amsmath}
\usepackage{amssymb,epsfig}
\usepackage{caption}
\usepackage{subcaption}
\usepackage{comment}
\usepackage[normalem]{ulem}
\usepackage{bbold}

\usepackage{amsthm}
\usepackage[absolute]{textpos}
\usepackage{mathtools,emptypage}

\usepackage[bookmarks=true]{hyperref}
\usepackage{xcolor}
\hypersetup{
    colorlinks,
    linkcolor={red!50!black},
    citecolor={blue!50!black},
    urlcolor={blue!80!black},
}

\usepackage{todonotes}

\mathtoolsset{showonlyrefs} % uncomment to tag only the referred equations (post-production) 

\newcommand{\N}{\mathbb{N}}

\newcommand{\R}{\mathbb{R}}
\newcommand{\Z}{\mathbb{Z}}

\newcommand{\sfd}{{\sf d}}

\def\Xint#1{\mathchoice
{\XXint\displaystyle\textstyle{#1}}%
{\XXint\textstyle\scriptstyle{#1}}%
{\XXint\scriptstyle\scriptscriptstyle{#1}}%
{\XXint\scriptscriptstyle\scriptscriptstyle{#1}}%
\!\int}
\def\XXint#1#2#3{{\setbox0=\hbox{$#1{#2#3}{\int}$ }
\vcenter{\hbox{$#2#3$ }}\kern-.6\wd0}}

\def\dashint{\Xint-}

\renewcommand{\d}{{\mathrm d}}

\newcommand{\restr}[1]{\lower3pt\hbox{$|_{#1}$}}

\newcommand{\X}{{\rm X}}
\newcommand{\Y}{{\rm Y}}

\makeatletter
\newcommand{\mytag}[2]{%
  \text{#1}%
  \@bsphack
  \begingroup
    \@onelevel@sanitize\@currentlabelname
    \edef\@currentlabelname{%
      \expandafter\strip@period\@currentlabelname\relax.\relax\@@@%
    }%
    \protected@write\@auxout{}{%
      \string\newlabel{#2}{%
        {\color{black}#1}%
        {\thepage}%
        {\@currentlabelname}%
        {\@currentHref}{}%
      }%
    }%
  \endgroup
  \@esphack
}
\makeatother

\allowdisplaybreaks

\newtheorem{theorem}{Theorem}[section]

\newtheorem{corollary}[theorem]{Corollary}
\newtheorem{lemma}[theorem]{Lemma}
\newtheorem{proposition}[theorem]{Proposition}
\theoremstyle{definition}
\newtheorem{definition}[theorem]{Definition}

\newcounter{Counter}

\newcommand{\CBA}{{\sf CBA}}
\newcommand{\GCBA}{{\sf GCBA}}
\newcommand{\CAT}{{\sf CAT}}

\title[Sobolev-to-Lipschitz property of GCBA spaces]{Sobolev-to-Lipschitz property of geodesically complete spaces with curvature bounded from above}
\author[Emanuele Caputo]{Emanuele Caputo}\address[Emanuele Caputo]{
Università di Trento, Dipartimento di Matematica, Via Sommarive 14, 38123
Trento, Italy}\email{emanuele.caputo@unitn.it}

\author[Nicola Cavallucci]{Nicola Cavallucci}\address[Nicola Cavallucci]{
University of Fribourg, Department of Mathematics,  Chemin du Mus\'ee 23,  1700 Fribourg, Switzerland}\email{n.cavallucci23@gmail.com}

\author[Toni Ikonen]{Toni Ikonen}\address[Toni Ikonen]{University of Fribourg, Department of Mathematics,  Chemin du Mus\'ee 23,  1700 Fribourg, Switzerland}
\email{toni.ikonen@unifr.ch}


\begin{document}

\maketitle

\begin{abstract}
    We prove that every length space with curvature bounded from above that is geodesically complete has the Sobolev-to-Lipschitz property with exponent infinity. That is, every Sobolev map in the $W^{1,\infty}$-space has a Lipschitz representative so that the Lipschitz constant coincides with the infinity energy of the map. The proof is geometric and relies on arbitrarily small perturbations of geodesics to a curve that has zero length on the singular set. The motivation is to develop the analytic theory of such spaces; in particular, our result implies that GCBA spaces satisfy the infinity Poincaré inequality and an essential assumption in the theory of Lipschitz--Volume rigidity.
\end{abstract}

\section{Introduction}

Metric spaces with curvature bounded from above ($\CBA$) have been introduced by Alexandrov in \cite{Alexandrov1951}, in terms of comparison of triangles with the simply connected two-dimensional spaces of constant sectional curvature. The flexibility of this class prevents a comprehensive study of their local geometry, see (\cite[Example 4.2]{LytchakNagano2019_GCBA}) and the discussion in their introduction. 
For these reasons, Lytchak--Nagano studied in \cite{LytchakNagano2019_GCBA} the local geometry of a suitable subclass, namely locally compact, locally geodesically complete spaces with curvature bounded from above ($\GCBA$). Geodesically complete means that every local geodesic can be extended to a longer local geodesics from both ends. Under these assumptions, for every point $x$ in a $\GCBA$ space $(\X,\sfd)$, there exists an integer $n_x\in \N$ such that the Hausdorff dimension of every small enough ball around $x$ is exactly $n_x$ (\cite[Theorem 1.2]{LytchakNagano2019_GCBA}). The quantity $n_x$ is called the \emph{dimension} of $x$. For $k\in \N$, we denote by $\X^k$ the set of points of dimension $k$. Then, for every $k\in \N$, the set $\X^k$ contains a dense and open subset $\mathcal{R}^k$ which is locally biLipschitz to $\R^k$ (\cite[Theorem 1.2]{LytchakNagano2019_GCBA}). The singular set is $\mathcal{S} := \bigcup_{k\in \N}\X^k \setminus \mathcal{R}^k$. With this stratification of $\X$, we consider the canonical measure
\begin{equation*}
    \mu \coloneqq \sum_{i=1}^\infty \mathcal{H}^k\llcorner{\X^k}
\end{equation*}
that is positive on open sets and locally finite (\cite[Theorem 1.4]{LytchakNagano2019_GCBA}). Moreover, $\mu(\mathcal{S}) = 0$, see \cite[Theorem 1.2]{LytchakNagano2019_GCBA}. We also note that the measure is \emph{infinitesimally doubling}:
\begin{equation*}
    \limsup_{ r \rightarrow 0^{+} } \frac{ \mu(B(x,2r)) }{ \mu(B(x,r)) } < \infty
    \quad\text{for almost every $x \in X$.}
\end{equation*}
We emphasize that for a connected $\GCBA$ space $\X$, the measure $\mu$ is (locally) doubling if and only if $\X$ is purely dimensional; see \cite[Theorem 5.2]{CavallucciSambusetti2020}.

The main goal of this paper is to initiate the study of analytical properties of a $\GCBA$ space as a metric measure space $(\X,\sfd,\mu)$. In general, several analytic properties are not well-understood. These include local properties such as local $(1,p)$-Poincaré inequalities (see \cite[Chapters 5-9]{HeinonenKoskelaShanmugalingamTyson2015}), other potential-theoretic tools such as capacity estimates (see \cite{HeinonenKilpelainenMartio93,BjornBjorn}), or other strong connectivity properties of the space (see \cite{caputo-2026-geometric-characterization-of-pi-spaces,caputo-cavallucci-2025-a-geometric-approach-to-poincare-inequaltiy-and-minkowski-content-of-separating-sets} and references therein). Part of the difficulty in developing an analytic toolkit in this context, is the vastness of the class. However, it is clear that e.g. the validity of local Poincaré inequality depends highly on the local geometry of the $\GCBA$ space of interest. For instance, the wedge sum of $n$-planes is a $\GCBA$ space with a local $(1,p)$-Poincaré inequality valid only for $p > n$ --- see \cite[Section 6.12]{HeinonenKoskela} for related discussion. On the other hand, this leads to the natural question whether a local $(1,\infty)$-Poincaré inequality could be valid on the whole class (as raised in \cite[Remark 7.4]{GarciaBravoIkonenZhu2025}).
To get an idea on some basic examples, consider the following.
\begin{enumerate}
    \item $\GCBA$ is stable under gluings of $\GCBA$ spaces along isometric locally complete convex subsets, see \cite[Theorem II.11.3]{BridsonHaefliger}. For example, a gluing of $n$- and $n'$-dimensional planes along a $k$-plane is of non-constant dimension if $n \neq n'$ and there exists an $n$-dimensional singular set of Hausdorff dimension $k$.%  Related examples include the gluing of tori of genus one along a point or a geodesic segment, and their universal covers.
    \item The spherical suspension of the Poincaré three-sphere is a $\GCBA$ space. Indeed, the Poincaré sphere has a canonical metric that has curvature bounded from above by one, and dilating the metric suitably and taking the spherical suspension gives the desired example. We also recall that a $\CBA$ space that is a homology manifold is always $\GCBA$, see the proof of \cite[Proposition II.5.12]{BridsonHaefliger}.
    \item Every finite simplicial complex has a metric that makes it a $\sf{CBA}$ space, see \cite[Theorem 1]{Berestovskii1983}. If it has no free faces, it is $\GCBA$ by \cite[Proposition II.5.10]{BridsonHaefliger}.
    \item There is a $\GCBA$ space that admits no triangulations, see \cite[Example 2.7]{Nagano2000}.
\end{enumerate}

To formulate our main result, we recall that a metric measure space $(\X,\sfd,\mu)$ is an \emph{$\infty$-thick length space} if for every $x,y\in \X$ and $N\subseteq \X$ with $\mu(N) = 0$, it holds that
$$\sfd(x,y) = \inf\left\{ \ell(\gamma) \colon \gamma \textup{ joins } x \textup{ to } y,\, \ell(\gamma \cap N)=0 \right\}.$$
Our main theorem shows that a length $\GCBA$ space is, in fact, an $\infty$-thick length space. 

\begin{theorem}
\label{theo:intro_main_thick_length_space}
    Let $(\X,\sfd)$ be a length ${\sf GCBA}$ space. Then $(\X,\sfd,\mu)$ is an $\infty$-thick length space.
    In particular, the following holds:
    \begin{itemize}
        \item[(i)] (Sobolev-to-Lipschitz property) for every $u \in W^{1,\infty}(\X)$, there exists a Lipschitz representative $\tilde{u}$ such that the global Lipschitz constant $\textup{Lip}(\tilde{u})$ is equal to the essential supremum $\Vert g_u\Vert_\infty$ where $g_u$ is the minimal $\infty$-weak upper gradient;
        \item[(ii)] (Poincaré inequality) $(\X,\sfd,\mu)$ satisfies a local $(1,\infty)$-Poincar\'{e} inequality. 
    \end{itemize}
\end{theorem}
We refer to Section \ref{sec:preliminaries} for the terminology. However, we emphasize that Theorem \ref{theo:intro_main_thick_length_space} (i) is equivalent to being an $\infty$-thick length space by \cite[Theorem 1.4]{GarciaBravoIkonenZhu2025} while Theorem \ref{theo:intro_main_thick_length_space} (ii) follows from being an $\infty$-thick length space by \cite[Theorem 1.5]{GarciaBravoIkonenZhu2025}. Hence, being an $\infty$-thick length space is the main content of Theorem \ref{theo:intro_main_thick_length_space}. It implies strong connectivity properties for weakly convex subsets of $\GCBA$ spaces.
\begin{corollary}\label{corollary-gcba-thickness}
    Let $(\X,\sfd)$ be a $\GCBA$ space. Then any open and weakly convex subset of $\X$ is an $\infty$-thick length space.
\end{corollary}
Here a subset $C \subseteq \X$ is \emph{weakly convex} if every $x, y \in C$ can be joined with a geodesic that lies in $C$ and has length equal to that of $\sfd(x,y)$. In particular, the restriction $( C, \sfd|_{ C \times C }, \mu\llcorner{C} )$ is a geodesic length $\GCBA$ space. Basic examples include small enough balls in any $\GCBA$ space. It is also clear that the $\GCBA$ property is stable when passing to the intrinsic distance on an open and connected subset. Thus we have the following corollary.
\begin{corollary}\label{corollary-gcba-thickness}
    Let $(\X,\sfd)$ be a $\GCBA$ space. Then any connected and open subset of $\X$ equipped with the intrinsic distance and the restriction of the reference measure is an $\infty$-thick length space.
\end{corollary}

The earliest constructions of $\infty$-thick length spaces were studied by De Cecco--Palmieri in the context of Lipschitz manifolds in \cite{DeCeccoPalmieri88}. In the setting of metric measure spaces, the notion has previously appeared in the work of Juutinen--Shanmugalingam \cite{Juut:Shan:06} and Durand-Cartagena--Jaramillo--Shanmugalingam \cite{DurandJaramilloShanmugalingam2012,DurandShanmugalingamWilliams2012,durand-cartagena-jaramillo-shanmugalingam-2016-geometric-characterizations-of-p-poincare-inequalities-in-the-metric-setting} when developing the theory of $\infty$-harmonic functions and $(1,\infty)$-Poincaré inequalities in metric measure spaces, respectively. Moreover, it was recently proved \cite[Theorem 1.4]{GarciaBravoIkonenZhu2025} by the third author et al. that in a locally complete metric measure space equipped with an infinitesimally doubling measure, being $\infty$-thick length space is equivalent to Theorem \ref{theo:intro_main_thick_length_space} (i). Note that a closely related Sobolev-to-Lipschitz property is important in the local geometry of metric measure spaces satisfying synthetic Ricci curvature lower bounds (\cite{AmbrosioGigliSavare2014,AmbrosioGigliSavare2015}); see also applications to the structure of minimal discs \cite{CreutzSoultanis2020}. As far as the authors are aware, a version of the Sobolev-to-Lipschitz property was first introduced \cite[Theorem 1.5]{RocknerSchied1999} in the metric setting. 

The local Poincaré inequalities have played a critical role on the analysis of metric measure spaces since the works of Heinonen--Koskela \cite{HeinonenKoskela} and Cheeger \cite{Cheeger99}; see \cite{HeinonenKoskelaShanmugalingamTyson2015} for further discussion. As mentioned above, the $(1,\infty)$-Poincaré was introduced in \cite{DurandJaramilloShanmugalingam2012}. This is the weakest of the local $(1,p)$-Poincaré inequalities, and thus Theorem \ref{theo:intro_main_thick_length_space} can be understood as a natural start for developing local analysis on $\GCBA$ spaces.

The conclusion of Theorem \ref{theo:intro_main_thick_length_space} is also directly related to recent developments in the so-called \emph{Lipschitz-volume rigidity problem}, see \cite{Zuest2024,marti2025lipschitzvolumerigidityproblemmetric} for an overview of the topic.
The general setup is the following. Let $f\colon \X \to \Y$ be a $1$-Lipschitz surjective map between two metric spaces such that $0<\mathcal{H}^n(\X)= \mathcal{H}^n(\Y) <\infty$ for some $n\in \N$. The Lipschitz-volume rigidity problem asks for conditions on $\X$ and $\Y$ under which $f$ is an isometry. A necessary condition for the validity of the statement is that $(\Y,\sfd_\Y,\mathcal{H}^n)$ is an $\infty$-thick length space, as observed in \cite[Example 4.1]{Zuest2024}, even under very regular assumptions, namely when $\Y$ is a smooth $n$-manifold equipped with a length distance locally biLipschitz to $\R^n$. Theorem \ref{theo:intro_main_thick_length_space} will be used in a forthcoming paper by the first two authors and Marti \cite{CaputoCavallucciMarti2026} in establishing Lipschitz-volume rigidity with certain $\GCBA$ targets.

\subsection{Strategy of the proof}

We first perform a simple reduction by proving a local-to-global property of the thick length space property in any length space. This allows us to reduce the analysis to a tiny ball in a $\GCBA$ space, i.e. a sufficiently small ball which has compact closure and is $\CAT(\kappa)$ for some $\kappa \in \R$. We refer to Section \ref{subsec:GCBA} for more details.
Given a geodesic lying in a tiny ball, we define an iterative procedure to constructing a curve connecting the endpoints whose length is almost that of the geodesic and which have zero length in the singular set.
Then, we perturb the curve in the regular set by using the locally biLipschitz equivalence to the Euclidean geometry in such a way that the new curve spends zero length in a given $\mu$-negligible set. Since the resulting curve has zero length on the singular part, this results in the desired curve. Controlling the length of the perturbation is key in the proof of Theorem \ref{theo:intro_main_thick_length_space}.

The following example illustrates the importance of the role of the structure theory of the singular set of a $\GCBA$ spaces in Theorem \ref{theo:intro_main_thick_length_space}. Given a lower semicontinuous weight $\omega \colon \R^n \to [0,1]$ on $\R^n$, we define the \emph{weighted distance} $\sfd_\omega(x,y) = \inf\{ \int_{\gamma} \omega \,ds \colon \text{ $\gamma$ joins $x$ to $y$ }\}$. Let $E \subset \R^n$ be a compact Euclidean segment and $\omega = 1$ in $\R^n \setminus E$ and $\omega = c$ in $E$ for $c \in (0,1)$. Then $( \R^n, \sfd_\omega )$ is bi-Lipschitz equivalent to $\R^n$ through the identity map, locally isometric to $\R^n$ on $\R^n \setminus E$, but it is not an $\infty$-thick length space when equipped with the $n$-dimensional Hausdorff measure. Indeed, if $\ell(E)$ is the Euclidean length of $E$, its length under $\sfd_\omega$ is $c\ell(E)$ while any curve having the same end points and having zero length on $E$ has length at least $\ell(E)$. A similar example was considered by Züst in \cite[Example 4.1]{Zuest2024} (see \cite{ikonen-romney-2022-quasiconformal-geometry-and-removable-sets-for-conformal-mappings} for analysis of more general weighted distances and applications to complex analysis and the uniformization problem, when $n = 2$).

\subsection{Acknowledgements}
E. Caputo was supported by the European Union’s Horizon 2020 research and innovation programme
(Grant agreement No. 948021) and the Swiss National Science Foundation (grant 200021-204501 ‘Regularity of sub-Riemannian geodesics and applications’) and warmly thanks the University of Fribourg for the collaborative working atmosphere. T. Ikonen was supported by the Swiss National Science Foundation grant 212867.

\section{Preliminaries}
\label{sec:preliminaries}
Let $(\X,\sfd)$ be a metric space, $x\in \X$ and $r>0$. We denote by $B(x,r)$ (resp. $\overline{B}(x,r)$) the \emph{open} (resp. \emph{closed}) \emph{ball} centered at $x$ and of radius $r$. A \emph{curve} is a continuous map $\gamma \colon [a,b] \to \X$. The starting and end points of $\gamma$ are denoted by $\alpha(\gamma) := \gamma(a)$ and $\omega(\gamma) := \gamma(b)$, respectively. We say that $\gamma$ joins $\alpha(\gamma)$ to $\omega(\gamma)$. The \emph{length} of $\gamma$ is
\begin{equation*}
    \ell(\gamma):= \sup \left\{ \sum_{i=0}^{N-1} \sfd(\gamma(t_i), \gamma(t_{i+1})), \,\{ t_i\}_{i=1}^N\subseteq [a,b],\,N \in \mathbb{N},\,
    a=t_0 < t_1 <\dots < t_N =b\right\}.
\end{equation*}
A continuous curve $\gamma\colon [a,b] \to \X$ is said to be \emph{rectifiable} if $\ell(\gamma)<\infty$. The curve $\gamma$ has \emph{constant-speed $\ell$} if $\ell( \gamma|_{ [s,t] } ) = \ell |t-s|$ for every $a \leq s < t \leq b$. Every rectifiable curve has a constant-speed reparametrization. A rectifiable curve is a \emph{length-minimizer} if $\ell(\gamma) = \sfd(\alpha(\gamma), \omega(\gamma))$. We say that a curve $\gamma$ is a \emph{geodesic} if $\sfd( \gamma(s), \gamma(t) ) = |t-s|$ for every $a \leq s < t \leq b$. Notice that every geodesic is a length-minimizer but not vice versa --- however, a unit speed reparametrization of a length-minimizing curve is a geodesic. A rectifiable curve $\gamma$ is a \emph{local geodesic} if for every $t \in [a,b]$, there exists $\varepsilon > 0$ such that $\sfd( \gamma(t'), \gamma(t'') ) = |t''-t'|$ for every $t',t'' \in [a,b]$ and $|t'-t|+|t''-t| \leq \varepsilon$. A local geodesic is not necessarily a length-minimizer. 

The \emph{concatenation} of two curves $\gamma \colon [a_1,b_1]\to \X$ and $\eta \colon [a_2,b_2]\to \X$ with $\eta(a_2)=\gamma(b_1)$ is the curve $\gamma \star \eta \colon [a_1,b_1+b_2-a_2]$ defined as $(\gamma \star \eta)(t) = \gamma(t)$ if $t \in [a_1,b_1]$ and $(\gamma \star \eta)(t)=\eta(t-(b_1-a_1))$ if $t \in [b_1,b_1+b_2-a_2]$. The concatenation is an associative operation and we will use this property below without further mention.
Let $(a,b)=\cup_{i\in \Z} [a_i,b_i]$ with $b_i=a_{i+1}$ for every $i \in \mathbb{Z}$. Let $\gamma_i \colon [a_i,b_i] \to \X$ be a sequence of curves such that $\gamma_i(b_i)=\gamma_{i+1}(a_{i+1})$ for $i \in \Z$. Then the \emph{infinite concatenation} of the family $\{\gamma_i\}_{i \in \Z}$ is given by the curve $\eta \colon (a,b)\to \X$ defined by $\eta(t)=\gamma_i(t)$ if $t \in [a_i,b_i]$.

A metric space $(\X,\sfd)$ is a \emph{length space} if $\sfd(x,y)$ coincides with the infimum over the length of curves joining $x$ to $y$, for every $x,y\in \X$. The space is \emph{$R$-geodesic} if for every $x,y\in \X$ with $\sfd(x,y)\le R$, there exists a geodesic joining $x$ to $y$. 
It is \emph{geodesically complete} if every local geodesic $\gamma \colon [a,b] \to \X$ can be extended to a local geodesic $\gamma \colon [a-\varepsilon,b+\varepsilon] \to \X$ for some $\varepsilon >0$. A subset $U \subset \X$ is \emph{weakly $R$-convex} if, for every $x,y \in U$ with $\sfd(x,y)\leq R$, there exists a geodesic in $U$ joining $x$ to $y$; in other words, it is weakly $R$-convex if $(U,\sfd\restr{U \times U})$ is $R$-geodesic. If also every such geodesic lies in $U$, then $U$ is \emph{$R$-convex}. We omit $R$ from the notation if $R = \infty$.

A \emph{metric measure space} is a triple $(\X,\sfd,\mu)$ where $(\X,\sfd)$ is a metric space and $\mu$ is a nonnegative Borel measure which is locally finite. Given $B \subseteq \X$ Borel, we define $\mu\llcorner{B}(A):=\mu(A\cap B)$ for every Borel $A \subseteq \X$, which is a nonnegative locally finite Borel measure on $B$.

We denote by $\mathcal{H}^k$ the $k$-dimensional Hausdorff measure on a metric space $(\X,\sfd)$. %Moreover, ${\rm dim}_H(\X)$ denotes the Hausdorff dimension of $\X$.

\subsection{Analysis on metric spaces}

Let $u \colon \X \to \R$ be a Lipschitz function. We denote by ${\rm Lip}(u)$ the Lipschitz constant of $u$, defined as $${\rm Lip}(u):=\sup_{y\neq x}\frac{|u(x)-u(y)|}{\sfd(x,y)}.$$

Given a function $u \colon \X \to \R$, a Borel function $g \colon \X \to [0,\infty]$ is an \emph{upper gradient} of $u$ if $$|u(\alpha(\gamma))-u(\omega(\gamma))|\le \int_\gamma g \,\d s$$ for any rectifiable curve $\gamma$. The integral on the right is the \emph{path integral of $g$ over $\gamma$}. By definition, the value is set to be infinite if $\gamma$ is not rectifiable and set to be zero if $\gamma$ is a constant curve. For non-trivial rectifiable curves, the value is defined by integration over the \emph{length measure} of $\gamma$ on its domain; we refer to \cite[Chapter 5]{HeinonenKoskelaShanmugalingamTyson2015} for details.

Let $(\X,\sfd,\mu)$ be a metric measure space.
For a function $u \colon \X \to \R$, an \emph{$\infty$-weak upper gradient of $u$} is an element in the $L^{\infty}(\mu)$-completion of the $L^{\infty}(\mu)$-classes of essentially bounded upper gradients of $u$. Standard arguments show that the set of $\infty$-weak upper gradients of $u$ is a closed lattice, cf. \cite[Lemma 4.1 and Proposition 4.5]{maly-2013-minimal-weak-upper-gradeints-in-newtonian-spaces-based-on-quasi-banach-function-lattices}, so there exists a unique element $g_u\in L^\infty(\mu)$, called the \emph{minimal $\infty$-weak upper gradient of $u$}, such that $g_u \le g$ $\mu$-a.e.\ for any $\infty$-weak upper gradient $g$.

We say that an equivalence class $u \in L^{\infty}(\mu)$ belongs to $W^{1,\infty}(\X)$ if it has a representative with an $\infty$-weak upper gradient, and we review basic properties based on \cite{maly-2013-minimal-weak-upper-gradeints-in-newtonian-spaces-based-on-quasi-banach-function-lattices}. If there are two representatives of the same equivalence class in $L^{\infty}(\mu)$ with an $\infty$-weak upper gradient, then the minimal $\infty$-weak upper gradient of the difference is zero and the minimal $\infty$-weak upper gradients coincide. This gives well-posedness for the definition of $W^{1,\infty}(\X)$ and we may equip it with the norm
\begin{align*}
    \|u\|_{ W^{1,\infty}(\X) } = \|u\|_{L^{\infty}(\mu)} + \|g_u\|_{ L^{\infty}(\mu) }.
\end{align*}
Standard arguments show that $W^{1,\infty}(\X)$ is a Banach space.

We say that a metric measure space $(\X,\sfd,\mu)$ satisfies the \emph{Sobolev-to-Lipschitz property} if every $u \in W^{1,\infty}(\X)$ has a Lipschitz representative $\tilde{u}$ such that ${\rm Lip}(\tilde{u})=\|g_u\|_{L^\infty(\mu)}$. We say that $(\X,\sfd,\mu)$ satisfies a local $(1,\infty)$-Poincar\'{e} inequality if there exist constants $C > 0$, $\lambda \ge 1$ such that 
\begin{equation*}
    \dashint_{B(x,r)} \Big| u(y) -\dashint_{B(x,r)} u \,\d \mu \Big|\,\d \mu(y) \le C r \| g \|_{L^\infty(\mu\llcorner B(x,\lambda r) )}
\end{equation*}
for any ball $B(x,r) \subseteq \X$ with $\mu( B(x,r) ) < \infty$ and every $u \colon X \to \R$ integrable on $B(x,r)$ with an $\infty$-upper gradient $g \colon \X \to \R$. This coincides with the definition introduced in \cite[Definition 1.2]{DurandJaramilloShanmugalingam2012}.

\subsection{Very thick quasiconvexity}
We recall very $\infty$-thick convexity that was introduced in \cite{durand-cartagena-jaramillo-shanmugalingam-2016-geometric-characterizations-of-p-poincare-inequalities-in-the-metric-setting} and its localization from \cite{GarciaBravoIkonenZhu2025}. To this end, let $(\X,\sfd)$ be a metric space, $N\subseteq \X$ be any set and consider a rectifiable curve $\gamma \colon [a,b] \to \X$. We say that $\gamma$ has \emph{zero/negligible length in $N$} if the preimage $\gamma^{-1}(N)$ is negligible for the length measure of $\gamma$; we denote this by $\ell( \gamma \cap N ) = 0$. We recall that the length measure of $\gamma$ on $[a,b]$ is an outer measure obtained from a Carathéodory's construction using the length functional $\nu([s,t]):=\ell(\gamma\restr{[s,t]})$. Hence the definition is well-posed. When $N$ is Borel, this corresponds to the path integral $\int_\gamma \chi_N \,\d s=0$.  %Using an outer measure, it is possible to consider negligible sets. Using this, we say that $\gamma$ has \emph{zero length on $N \subseteq \X$} if the preimage $\gamma^{-1}(N)$ is negligible with respect to the length measure.

\begin{definition}[Very thick quasiconvexity]
\label{definition-thick-quasiconvexity}
Let $C \geq 1$ and $r > 0$. Let $(\X, \sfd, \mu )$ be a metric measure space. We say that $\X$ is \emph{very $\infty$-thick $(C,r)$-quasiconvex at $x$} if for every $y, z \in B( x, r )$ and every $\mu$-negligible set $N \subseteq \X$, there exists a rectifiable curve $\gamma$ joining $y$ to $z$ such that $\ell( \gamma ) \leq C \sfd(y,z)$ and $\ell(\gamma \cap N)=0$. We say that $\X$ is \emph{very $\infty$-thick $C$-quasiconvex} if we may take $r = \infty$.
\end{definition} 
We recall that a metric measure space is an $\infty$-length space if it is very $\infty$-thick $(1+\varepsilon)$-quasiconvex for every $\varepsilon > 0$.
We have the following local-to-global property of very $\infty$-thick quasiconvexity.

\begin{lemma}
\label{lemma-bootstrapping-the-local-qc-property}
    Let $(\X,\sfd,\mu)$ be a metric measure space and $\varepsilon > 0$. Suppose that for every $x \in \X$, there exists $r_x > 0$ such that $\X$ is very $\infty$-thick $(1+\varepsilon,r_x)$-quasiconvex at $x$. Then
    each connected component of $\X$, endowed with the length distance and the restriction of the measure, is very $\infty$-thick $(1+\varepsilon+\varepsilon')$-quasiconvex for every $\varepsilon' > 0$.
    Moreover, if $U \subseteq \X$ is an open weakly convex set, then $U$ is very $\infty$-thick $(1+\varepsilon)$-quasiconvex.
\end{lemma}
\begin{proof}
    We will prove the following claim: for every rectifiable curve $\theta \colon [0,1] \to X$ and every negligible set $N \subseteq \X$, there exists a curve $\gamma \colon [0,1] \to X$ with the same endpoints, with $\ell(\gamma) \leq (1+\varepsilon)\ell(\theta)$, and $\gamma$ having negligible length in $N$. We later show how the statements of the lemma follows.

    Consider a non-constant rectifiable curve $\theta \colon [0,1] \to \X$. We cover $K = \theta( [0,1] )$ by the balls $\{ B( x, r_x ) \}_{ x \in K }$. By Lebesgue's number lemma, there exists $k \in \mathbb{N}$ such that if $0 \leq s < t \leq 1$ satisfy $t-s \leq \frac{1}{k}$, then $\theta( [s,t] ) \subseteq B( x, r_x )$ for some $x \in K$. Let $N \subseteq \X$ be $\mu$-negligible and consider a partition $t_j = j/k$ for $j = 0, \dots, k$ of $[0,1]$. For each $j = 0, \dots, k-1$, we consider a rectifiable curve $\gamma_j \colon [t_j, t_{j+1}] \to \X$ joining $\theta( t_j )$ to $\theta( t_{j+1} )$ with $\ell( \gamma_j ) \leq (1+\varepsilon) \sfd(\theta(t_j),\theta(t_{j+1}))$ and having zero length in $N$. The curve $\gamma \colon [0,1] \to \X$ defined by $\gamma :=\gamma_0 \star \gamma_1 \cdots \star \gamma_{k-1}$ has zero length in $N$ and 
    $\ell( \gamma ) =\sum_{j=0}^{k-1} \ell(\gamma_j)\leq (1+\varepsilon)\sum_{j=0}^{k-1} \sfd(\theta(t_j),\theta(t_{j+1})) \leq (1+\varepsilon) \sum_{j = 0}^{k-1} \ell( \theta|_{ [t_j,t_{j+1}] } ) = (1+\varepsilon)\ell(\theta)$. This proves the claim.

    Now, we show that each connected component with the length distance has the stated property. For this, fix $x, y \in \X$ in a connected component and a $\mu$-negligible set $N \subset \X$. 
    From the claim, it follows that, denoting by $\sfd_\ell$ the intrinsic distance on a connected component $C$, we have
    \begin{equation*}
        \inf\{ \ell(\gamma):\, \alpha(\gamma)=x,\,\omega(\gamma)=y,\, \ell(\gamma \cap N)=0\}\le (1+\varepsilon) \sfd_\ell(x,y).
    \end{equation*}
    The statement about the weakly convex open set follows by a similar concatenation argument as above.%It remains to verify the claim about the weakly convex open set $U$: consider a geodesic $[x,y] \subset U$ and partition the geodesic into segments $I_0 = [x,x_1], I_1 =  [x_1,x_2], \dots, I_j = [x_j,x_{j+1}], \dots I_{n} = [x_n,x]$ such that the $(2\varepsilon) \ell( I_j )$-neighbourhood of $I_j$ is contained in $U$. Now replace is segment $I_j$ by the curve from the claim and concatenate the curves to conclude.
\end{proof}

\subsection{$\GCBA$ spaces}
\label{subsec:GCBA}

Given $\kappa \in \R$, we define $D_\kappa := \pi/\sqrt{\kappa}$ when $\kappa >0$ and $D_\kappa = \infty$ otherwise. A complete metric space $(\X,\sfd)$ is $\CAT(\kappa)$ if it is $D_\kappa$-geodesic and every geodesic triangle of perimeter at most $2D_\kappa$ is smaller than its comparison triangle in the simply connected, $2$-dimensional Riemannian manifold with sectional curvature $\kappa$. We refer to \cite{BridsonHaefliger} for basic properties of such spaces.

A metric space $(\X,\sfd)$ is a $\GCBA$ space if it is separable, geodesically complete, locally compact and if for every point $x\in \X$ there exist $\kappa_x\in \R$ and $0< r_x \le D_{\kappa_x}$ such that $\overline{B}(x,r_x)$ is $\CAT(\kappa_x)$. A ball $B(x,r) \subseteq \X$ is called a \emph{tiny ball} if $\overline{B}(x,10r)$ is compact and $\textup{CAT}(\kappa)$ for some $\kappa \in \R$, having radius $r\le \min\{1,D_\kappa/100\}$. Let $x\in \X$. The {\em space of directions} at $x$ is 
$$\Sigma_x\X := \{\gamma\colon [a,b] \to \X\,:\, \gamma \text{ geodesic with } \alpha(\gamma)=x\}/\sim,$$
where $\gamma \sim \gamma'$ if and only if $\angle_x(\gamma,\gamma') = 0$; for the definition of the {\em Alexandrov angle} between two   geodesics, we refer to \cite[Definition I.1.12]{BridsonHaefliger}. We refer to the resulting equivalence class $[\gamma]$ as the direction of $\gamma$ (at $x$).
The space $\Sigma_x\X$ equipped with the distance $\angle_x(\cdot,\cdot)$ is a compact, geodesically complete, $\CAT(1)$ space of diameter $\pi$ if $x$ is not isolated (\cite[Corollary 5.9]{LytchakNagano2019_GCBA}). The Euclidean cone over the space of directions, namely $T_x\X := \textup{Cone}(\Sigma_x\X) = [0,+\infty)\times \Sigma_x\X /_\sim$, $(0,v)\sim(0,w)$ for every $v,w\in \Sigma_x\X$, with the corresponding metric $\sfd_T$ (see \cite[Definiton I.5.6]{BridsonHaefliger}), is a proper, geodesically complete, $\CAT(0)$ space which is called the {\em tangent space} at $x$. The tip of the cone $T_x\X$ is denoted by $o_x$. 
The triple $(T_x\X,\sfd_T,o_x)$ is the Gromov-Hausdorff limit of the sequence of pointed metric spaces $(\X,\lambda_i\sfd,x)$, for every $\lambda_i \to \infty$ (see \cite[Corollary 5.7]{LytchakNagano2019_GCBA} and \cite[Lemma 2.5]{CavallucciSambusetti2020}). This is proved as follows: for every tiny ball $B(x,r)$, there is a well-defined logarithmic map $\log_x\colon B(x,r) \to T_x\X$, $\log_x(y) := (\sfd(x,y),[\gamma_{xy}])$, where $\gamma_{xy}$ is the unique geodesic joining $x$ to $y$, that satisfies the property stated in the subsequent lemma. 

\begin{lemma}[{\cite[Lemma 5.5]{LytchakNagano2019_GCBA}}]
\label{lemma:logarithmic_map_convergence_to_the_tangent}
    Let $(\X,\sfd)$ be a $\GCBA$ space and let $x\in \X$. For every $\varepsilon > 0$, there exists $0<r_\varepsilon(x)<1$ such that for every $0< r\le r_\varepsilon(x)$ and every $x_1,x_2\in B(x,r)$, it holds that
    $$\vert \sfd(x_1,x_2) - \sfd_T(\log_x(x_1),\log_x(x_2)) \vert \le \varepsilon r,$$
    and that $\log_x( \overline{B}(x,r) ) = \overline{B}(o_x,r)$.
\end{lemma}

Moreover, the ball $B(o_x,r) \subseteq T_x\X$ is a tiny ball for every $r\le 1$.

We introduce further standard terminology. Let $x\in \X$, $k\in \N$ and $\delta > 0$. We say that a $k$-tuple $\{v_1,\ldots,v_k\} \subseteq \Sigma_x\X$ is \emph{$\delta$-spherical} if there exists another $k$-tuple $\{\bar{v}_1,\ldots,\bar{v}_k\} \subseteq \Sigma_x\X$ such that:
\begin{itemize}
    \item for every $i=1,\ldots,k$, $\angle_x(v_i,w) + \angle_x(w,\bar{v}_i) < \pi + \delta$ for every $w\in \Sigma_x\X$;
    \item for every $1\le i\ne j \le k$, $\angle_x(v_i,v_j), \angle_x(v_i,\bar{v}_j), \angle_x(\bar{v}_i,\bar{v}_j) < \pi/2 + \delta$.
\end{itemize}
A point $x\in \X$ is said to be \emph{$(k,\delta)$-strained} if $\Sigma_x\X$ has a $\delta$-spherical $k$-tuple. For each open set $U \subset \X$, we respectively define the $(k,\delta)$-\emph{regular points} and the $\delta$-\emph{regular points} as
    $$\mathcal{R}^k_{\delta}(U) := \{x\in \X^k\cap U\,:\, \text{$x$ is $(k,\delta)$--strained}\} \quad\text{and}\quad \mathcal{R}_\delta(U) := \bigcup_{k\in \N} \mathcal{R}_\delta^k(U).$$ 
The following proposition implies, in particular, that the $\delta$-regular points are open and dense in any tiny ball for sufficiently small $\delta > 0$, and they are closely related to the regular points of the tangent cone.
\begin{proposition}
\label{prop:properties_tiny_balls}
    Let $(\X,\sfd)$ be a $\GCBA$ space and let $U\subseteq \X$ be a tiny ball. Then there exists $\delta_0 > 0$ 
    %chosen as in \cite[Section 11, page 324]{LytchakNagano2019_GCBA} depending only on $U$, 
    such that the following holds for every $0<\delta \le \delta_0$ and for every $k\in \N$.
\begin{itemize}
	\item[(i)] For every tiny ball $V$ contained in $U$, the set $\mathcal{R}^k_\delta(V)$ is open in $\X$ and $\mathcal{R}_\delta(V)$ is dense in $V$. The corresponding statements hold for every tiny ball in $T_x\X$ and $\Sigma_x\X$ for each $x \in U$.
    \item[(ii)] For every $x\in U$ and $v \in \mathcal{R}_{\delta}^{k-1}(\Sigma_x\X )$, there exists $0<\rho(x,v) <1$ with the following property: if $\gamma$ is a geodesic in $\X$ with $\alpha(\gamma) = x$ and $[\gamma] = v$, then $\gamma(t) \in \mathcal{R}_\delta^{k}(U)$ for every $0<t\le \rho(x,v)$.
    \item[(iii)] There exists a function $\eta \colon ( 0, \delta_0) \rightarrow (0,\infty)$, with $\lim_{\delta \to 0}\eta(\delta) = 0$, satisfying the following property. For every $0<\delta \le \delta_0$ and for every connected subset $B \subseteq \mathcal{R}_{\delta}(U)$, there exists $k \in \N$ such that, for every $x\in B$, there exist $r=r(x) > 0$ and a $(1+\eta(\delta))$-bi-Lipschitz open map $F_x \colon B(x,r) \to \mathbb{R}^k$.
\end{itemize}
\end{proposition}

    In a $\GCBA$ space $(\X,\sfd)$, the tangent cone $T_x\X$ at every point $x\in \X$ is a metric cone. This implies that, for every $v\in \Sigma_x\X$, it holds that $T_{(t,v)}( T_x\X ) \cong \R \times T_v(\Sigma_x\X)$ for each $(t,v) \in T_x\X \setminus \{o_x\}$, see \cite[Lemma 2.4]{CavallucciSambusetti2026}.
    By definition, if $v\in \mathcal{R}^{k-1}_\delta(\Sigma_x\X)$, then $(t,v) \in \mathcal{R}^k_\delta(T_x\X)$ for every $t>0$. Since $\mathcal{R}_{\delta}( \Sigma_x \X)$ is dense in $\Sigma_x\X$ for small enough $\delta > 0$ by Proposition \ref{prop:properties_tiny_balls}.(i),  Proposition \ref{prop:properties_tiny_balls}.(ii) applies to a dense, open set of directions $v\in \Sigma_x\X$.
\begin{proof}
    The tiny ball $U$ is $N$-doubling in the sense of \cite[Definition 5.2]{LytchakNagano2019_GCBA}, see \cite[Proposition 5.1]{LytchakNagano2019_GCBA}. Moreover, for every $x \in U$, every tiny ball in $T_x\X$ is $N$-doubling by \cite[Corollary 5.8]{LycthakNagano2021} and every tiny ball in $\Sigma_x\X$ is $N_1$-doubling for $N_1=N_1(N)$ by \cite[Corollary 5.9]{LycthakNagano2021}. We choose $\delta_0$ as in \cite[Pag. 324]{LytchakNagano2019_GCBA}, depending only on $\max\{N,N_1\}$.
    Let $0 < \delta \leq \delta_0$. Statement (i) is immediate from \cite[Corollary 11.8]{LytchakNagano2019_GCBA}.
    
    We prove (ii). We fix $x,v$ as in the hyphothesis and we define $d_x := \sfd(x,U^c) > 0$. Suppose by contradiction that the thesis is false. Hence, we may find a sequence of positive real numbers $\{ s_j \}_j$ converging to $0$ and a sequence of geodesics $\gamma_j$ with $\alpha(\gamma_j) = x$ and $[\gamma_j] = v$ such that $\gamma_j(s_j/2) \notin \mathcal{R}^k_\delta(U)$. We may suppose that $s_j < d_x$ for every $j$. We consider the open balls $U_j$ centered at $x$ and of radius $1$ in the pointed metric space $(\X, \frac{1}{s_j}\sfd,x)$. By the previous paragraph, we are in the standard setting of convergence as defined in \cite[Definition 5.12]{LytchakNagano2019_GCBA}. The balls $U_j$ converge in the Gromov-Hausdorff sense to $B(o_x,1) \subseteq T_x\X$, in such a way that the logarithmic map defines one of the quasi-isometries; cf. Lemma \ref{lemma:logarithmic_map_convergence_to_the_tangent} and \cite[Lemma 2.5 and comment below]{CavallucciSambusetti2020}. In particular, $\gamma_j(s_j/2)$ converges to $(\frac{1}{2},v)$. We note that $(\frac{1}{2},v) \in \mathcal{R}_\delta^k(T_x\X)$ by the assumption on $v$ and the discussion before the proof. We are in position to apply \cite[Lemmas 11.5, 11.7 and 7.8]{LytchakNagano2019_GCBA}. These imply that $\gamma_j(s_j/2) \in \mathcal{R}^k_\delta(U)$ for every sufficiently large $j$. This is a contradiction, so (ii) follows.

    We prove item (iii). For every $k\in \N$, the set $\mathcal{R}_\delta^k(U)$ is open by
    (i). 
    Clearly $B$ is the union of the disjoint open sets $B\cap \mathcal{R}^k_\delta(U)$ so $B = B\cap \mathcal{R}^k_\delta(U)$ for some $k$ by connectedness of $B$. The thesis follows from \cite[Corollary 11.2]{LytchakNagano2019_GCBA} because of the choice of $\delta_0$.
\end{proof}

Given a tiny ball $U\subseteq \X$ and $\delta>0$, we define the set of $\delta$-singular points of $U$ as $\mathcal{S}_\delta(U) := U \setminus \mathcal{R}_\delta(U)$. By \cite[Corollary 11.8]{LytchakNagano2019_GCBA}, we have
\begin{equation}
    \label{eq:zero_measure_singular_set}
    \mu(\mathcal{S}_\delta(U)) = 0.
\end{equation}

%\begin{proposition}
%\label{prop:locally_biLip}
%    Let $(\X,\sfd)$ be a $\GCBA$ space, $U\subseteq \X$ be a tiny ball, and $\delta_0 >0$ be the constant of Proposition \ref{prop:properties_tiny_balls}. Then there exists a function $\eta \colon ( 0, \delta_0) \rightarrow (0,\infty)$, with $\lim_{\delta \to 0}\eta(\delta) = 0$, satisfying the following property. For every $0<\delta \le \delta_0$ and for every connected subset $B \subseteq \mathcal{R}_{\delta}(U)$, there exists $k \in \N$ such that, for every $x\in B$, there exist $r=r(x) > 0$ and a $(1+\eta(\delta))$-bi-Lipschitz open map $F_x \colon B(x,r) \to \mathbb{R}^k$.
%\end{proposition}

\section{Proof of the main theorem}
    We prove Theorem \ref{theo:intro_main_thick_length_space}. By Lemma \ref{lemma-bootstrapping-the-local-qc-property}, it is enough to prove that every tiny ball in a $\GCBA$ space is an $\infty$-thick length space. This will follow from the next results that will be proved in the forthcoming sections.

    \begin{theorem}
    \label{theorem-avoidance-irregular-part}
        Let $(\X,\sfd)$ be a $\GCBA$ space, $U\subseteq \X$ be a tiny ball and $\delta_0 >0$ be the constant of Proposition \ref{prop:properties_tiny_balls}. Then, for every $0<\delta\le \delta_0$, every $\varepsilon > 0$ and every $x,y\in U$, there exists a curve $\gamma$ joining $x$ to $y$ and contained in $U$ such that $\ell( \gamma ) \leq (1+\varepsilon) \sfd(x,y)$ and $\ell(\gamma \cap \mathcal{S}_{\delta}(U)) = 0$.
    \end{theorem}

    \begin{proposition}
    \label{proposition-regular-part}
        Let $(\X,\sfd)$ be a $\GCBA$ space, $U\subseteq \X$ be a tiny ball and $\varepsilon > 0$. Then there exists $\delta(\varepsilon)>0$ such that for every geodesic joining $x,y \in \mathcal{R}_{\delta(\varepsilon)}(U)$ in $\mathcal{R}_{\delta(\varepsilon)}(U)$ and every $\mu$-negligible subset $N\subseteq U$, there exists a curve $\gamma$ joining $x$ to $y$, contained in $\mathcal{R}_{\delta(\varepsilon)}(U)$, satisfying $\ell(\gamma) \le (1+\varepsilon)\sfd(x,y)$ and $\ell(\gamma \cap N) = 0$.
    \end{proposition}

    Theorem \ref{theorem-avoidance-irregular-part} allows us to find an almost geodesic curve connecting any two point in $U$ and spending zero length in the singular part. Proposition \ref{proposition-regular-part} gives a way to modify this curve in the regular part in order to spend zero length in any fixed negligible subset. We show how to conclude the proof of Theorem \ref{theo:intro_main_thick_length_space} with these results.

    \begin{proof}[Proof of Theorem \ref{theo:intro_main_thick_length_space}]
        According to the terminology of Definition \ref{definition-thick-quasiconvexity}, we have to prove that $(\X,\sfd)$ is very $\infty$-thick $(1+\varepsilon)$-quasiconvex for every $\varepsilon > 0$. Lemma \ref{lemma-bootstrapping-the-local-qc-property} implies that it is enough to prove that every tiny ball $U\subseteq \X$ is very $\infty$-thick $(1+\varepsilon)$-quasiconvex.

        We fix a tiny ball $U\subseteq \X$, with the associated $\delta_0>0$ provided by Proposition \ref{prop:properties_tiny_balls}. For every $\varepsilon > 0$, we define $\varepsilon' := (1+\varepsilon)^{\frac{1}{2}} - 1$ and we denote by $\delta(\varepsilon')>0$ the parameter obtained from Proposition \ref{proposition-regular-part}. We set $\delta_* := \min\{\delta_0,\delta(\varepsilon')\}>0$. Let $x,y\in U$ be arbitrary and let $N\subseteq \X$ be a $\mu$-negligible subset. Without loss of generality, we can suppose that $N\supseteq \mathcal{S}_{\delta_*}(U)$, by \eqref{eq:zero_measure_singular_set}.

        By Theorem \ref{theorem-avoidance-irregular-part}, we find a curve $\gamma \colon [0,1] \to \X$ joining $x$ to $y$ and contained in $U$ such that $\ell(\gamma) \le (1+\varepsilon')\sfd(x,y)$ and $\ell(\gamma \cap \mathcal{S}_{\delta_*}(U)) = 0$.
        By Proposition \ref{prop:properties_tiny_balls}.(i), the set $U\setminus\mathcal{S}_{\delta_*}(U) = \mathcal{R}_{\delta_*}(U)$ is open. Hence, $\gamma^{-1}(\mathcal{R}_{\delta_*}(U))$ is relatively open and, as such, a countable union of (relatively) open intervals. 
        
        Let $\gamma^{-1}(\mathcal{R}_{\delta_*}(U)) \setminus \{0,1\} = \bigcup_{i\in I} (a_i,b_i)$, where $I$ is a countable index set. For every $i\in I$, consider an increasing sequence $\{ t^j_i \}_{ j \in \Z }$ so that $\lim_{ j \to -\infty } t^j_i = a_i$, $\lim_{ j \to \infty } t^j_i = b_i$, and $\gamma( [t_{i}^{j},t_{i}^{j+1}] )$ is contained in a tiny ball contained in $\mathcal{R}_{\delta_{*}}(U)$ for every $i \in I$ and $j \in \Z$. By Proposition \ref{proposition-regular-part}, we find a curve $\theta^j_i$, mapping into $\mathcal{R}_{\delta_*}(U)$, and joining $\gamma(t^j_i)$ to $\gamma(t^{j+1}_i)$ such that $\ell(\theta^j_i) \le (1+\varepsilon')\sfd(\gamma(t^j_i),\gamma(t^{j+1}_i))$ and $\ell(\theta^j_i\cap N) = 0$.

        For every $i\in I$, we replace $\gamma\restr{(a_i,b_i)}$ by the infinite concatenation of the curves $\{ \theta^j_i \}_{ j \in \Z }$. We denote the resulting curve by $\tilde{\gamma}$, joining $x$ to $y$. By construction,
        \begin{equation*}
        \begin{aligned}
            \ell(\tilde{\gamma}) &\le \sum_{i\in I} \sum_{j\in \Z} (1+\varepsilon')\sfd(\gamma(t^j_i),\gamma(t^{j+1}_i)) \le (1+\varepsilon')\sum_{i\in I}\ell(\gamma\restr{(a_i,b_i)}) \\
            &\le (1+\varepsilon')^2\sfd(x,y) \le (1+\varepsilon)\sfd(x,y).
        \end{aligned}
        \end{equation*}
        Moreover,
        $$\ell(\tilde{\gamma}\cap N) = \sum_{i\in I} \sum_{j\in \Z} \ell(\theta_i^j\cap N) = 0.$$
        Theorem \ref{theo:intro_main_thick_length_space} (i) is immediate from \cite[Theorem 1.4]{GarciaBravoIkonenZhu2025}. Regarding the local $(1,\infty)$-Poincaré inequality, see \cite[Theorem 1.5]{GarciaBravoIkonenZhu2025}.
    \end{proof}

\subsection{Avoiding the singular part}

In this subsection, we prove Theorem \ref{theorem-avoidance-irregular-part}. For that, we will need the following covering lemma.

\begin{lemma}
\label{lemma:covering_geodesics_with_balls}
    Let $(\X,\sfd)$ be a metric space and let $\gamma\colon I \to \X$ be a geodesic. For every finite collection of balls $\{B(x_i,r_i)\}_{i=1}^M$ such that $\{x_i\}_{i} \subseteq \gamma(I)$ and $\gamma(I)\subseteq \bigcup_i B(x_i,r_i)$, there exists a subcollection $\{B(x_{i_j},r_{i_j})\}_{j=1}^N$ such that 
    \begin{itemize}
        \item[(i)] $\gamma^{-1}(x_{i_j}) < \gamma^{-1}(x_{i_{j+1}})$ for every $j\in \{1,\ldots,N\}$;
        \item[(ii)] $B(x_{i_j},r_{i_j})\nsubseteq B(x_{i_k},r_{i_k})$ for every $j,k\in \{1,\ldots,N\}$ with $j\neq k$.
    \end{itemize}
    In particular, the following properties hold:
    \begin{itemize}
        \item[(iii)] $\alpha(\gamma) \in B(x_{i_1},r_{i_1})$ and $\omega(\gamma) \in B(x_{i_N},r_{i_N})$;
        \item[(iv)] $r_{i_j} + r_{i_{j+1}} \ge \sfd(x_{i_j},x_{i_{j+1}})$ for every $j \in \{1,\ldots, N-1\}$.
    \end{itemize}
\end{lemma}

\begin{proof}

We relabel the balls $\{B(x_i,r_i)\}_{i=1}^M$ in such a way that $\gamma^{-1}(x_i) < \gamma^{-1}(x_{i+1})$ for every $i=1,\dots,M-1$, so (i) is satisfied. We define the subset $\{i_j\}\subseteq \{1,\ldots,M\}$ such that $\{B(x_{i_j},r_{i_j})\}_j$ satisfies (ii) as follows: 
$$i_1 := \max\{m \,:\, \alpha(\gamma) \in B(x_m,r_m)\},$$ 
and given $i_j$ such that $\bigcup_{\ell=1}^j B(x_{i_\ell}, r_{i_\ell}) \nsupseteq \gamma(I)$, we define
$$i_{j+1}:=\max\{m > i_j\,:\, B(x_m,r_m)\cap B(x_{i_j},r_{i_j}) \neq \emptyset\}.$$
The family $\{B(x_{i_j},r_{i_j})\}_j$ still covers $\gamma(I)$. Moreover, if $x\in \gamma(I)$ satisfies $x\in B(x_{i_{j_1}}, r_{i_{j_1}}) \cap B(x_{i_{j_2}}, r_{i_{j_2}}) \cap B(x_{i_{j_3}}, r_{i_{j_3}})$ with $i_{j_1} < i_{j_2} < i_{j_3}$, then the maximality in the definition of $i_{j_2}$ is violated. The new cover satisfies (ii) and still satisfies (i).
Any covering satisfying (i) and (ii) satisfies also (iii) and (iv).
\end{proof}

The subsequent proposition plays a key part in the recursive construction used in the proof of Theorem \ref{theorem-avoidance-irregular-part}.
\begin{proposition}
\label{prop:recursive_construction}
    Let $(\X,\sfd)$ be a $\GCBA$ space, let $U\subseteq \X$ be a tiny ball, and let $\delta_0 >0$ be the parameter from Proposition \ref{prop:properties_tiny_balls}. Then, for every $0<\delta\le \delta_0$, every $\varepsilon > 0$ and every $x,y\in U$, there exists a curve $\gamma_\varepsilon$ with the following properties:
    \begin{itemize}
        \item[(i)] $\alpha(\gamma_\varepsilon)=x$, $\omega(\gamma_\varepsilon) = y$, $\ell(\gamma_\varepsilon) \le (1+\varepsilon)\sfd(x,y)$ and $\gamma_\varepsilon$ is contained in the $\varepsilon\sfd(x,y)$-neighbourhood of the geodesic joining $x$ to $y$;
        \item[(ii)] the curve $\gamma_\varepsilon$ is the concatenation of a finite number of geodesic segments $\{\gamma_\varepsilon^j\}_{j\in J}$ of two families:
        \begin{equation*}
            \mathcal{G} := \{j \in J:\, \ell(\gamma_\varepsilon^j \cap \mathcal{S}_\delta(U))=0\}
            \quad\text{and}\quad
            \mathcal{B} := \{ j \in J:\, \ell(\gamma_\varepsilon^j \cap \mathcal{S}_\delta(U))>0\};
        \end{equation*}
        \item[(iii)] the latter family satisfies
        $
            \sum_{j\in \mathcal{B}} \ell(\gamma_\varepsilon^j) \le \varepsilon\sfd(x,y).
        $
    \end{itemize}
\end{proposition}

\begin{proof}
We fix $x,y\in U$, $0<\delta\le\delta_0$ and $\varepsilon > 0$. We set $\varepsilon' := \varepsilon/4$. Let $\gamma \colon [0,\sfd(x,y)] \to U$ be the constant-speed geodesic joining $x$ to $y$. For every $t\in (0,\sfd(x,y))$, we consider the parameter 
\begin{equation*}
 0< r_{\varepsilon'}(t) :=\min\{ t, \sfd(x,y)-t, r_{\varepsilon'}(\gamma(t))\} <1,   
\end{equation*}
where $r_{\varepsilon'}(\gamma(t))$ is provided by Lemma \ref{lemma:logarithmic_map_convergence_to_the_tangent}. We consider the points $(r_{\varepsilon'}(t),v^\pm_t) := \log_{\gamma(t)}(\gamma(t\pm r_{\varepsilon'}(t))) \in T_{\gamma(t)}\X$, where $v^\pm_t\in \Sigma_{\gamma(t)}\X$.
Moreover, for $t=0$ and $t=\sfd(x,y)$, we define $r_{\varepsilon'}(t)= \min\{r_{\varepsilon'}(\gamma(t)),\sfd(x,y)\}$, where $r_{\varepsilon'}(\gamma(t))$ is provided by Lemma \ref{lemma:logarithmic_map_convergence_to_the_tangent}, and below we consider the points $(r_{\varepsilon'}(0),v^+_0) := \log_{x}(\gamma(r_{\varepsilon'}(0))) \in T_{x}\X$ and $(r_{\varepsilon'}(\sfd(x,y)),v^-_{\sfd(x,y)}) := \log_{y}(\gamma(\sfd(x,y) - r_{\varepsilon'}(\sfd(x,y)))) \in T_{y}\X$, respectively. 

Let $t \in [0,\sfd(x,y)]$. By Proposition \ref{prop:properties_tiny_balls}.(i), we find points $w_t^\pm \in \mathcal{R}_{\delta}(\Sigma_{\gamma(t)}\X)$ such that $\angle_{\gamma(t)}(v_t^\pm, w_t^\pm) \le \varepsilon'$. Since $T_z\X$ is a metric cone for each $z \in \X$, we also have
\begin{equation}
    \sfd_T((r_{\varepsilon'}(t),w_t^\pm), (r_{\varepsilon'}(t),v_t^\pm)) \le \varepsilon' r_{\varepsilon'}(t).
\end{equation}

We apply Proposition \ref{prop:properties_tiny_balls}.(ii) to the pair $w_t^{\pm}$ and let $\rho(\gamma(t),w^{\pm}_t) > 0$ denote the given parameters.
Finally, we define $$r(t) := \min\{\rho(\gamma(t),w^+_t), \rho(\gamma(t),w^-_t), r_{\varepsilon'}(t)\} > 0.$$
Next, let $\eta^\pm_t$ be a geodesic with starting point $\alpha(\eta^\pm_t) = \gamma(t)$ and initial direction $[\eta^\pm_t] = w_t^\pm$. By Proposition \ref{prop:properties_tiny_balls}.(ii) and the cone property of $T_{\gamma(t)}\X$, it holds that
$$\eta^\pm_t(s) \in \mathcal{R}_{\delta}(U) \text{ for all } 0<s\le r(t),$$
and
$$\sfd_T((\log_{\gamma(t)}(\eta^\pm_t(s)), \log_{\gamma(t)}(\gamma(t \pm s)))) = \sfd_T((s, w^\pm_t), (s,v^\pm_t)) \le \varepsilon' s \quad \forall 0<s\le r(t).$$
By Lemma \ref{lemma:logarithmic_map_convergence_to_the_tangent}, we obtain that
\begin{equation}
    \label{eq:error_eta_regular}
    \sfd(\eta^\pm_t(s), \gamma(t \pm s)) \le 2\varepsilon' s \text{ for all } 0<s\le r(t).
\end{equation}
By compactness, we find $N$ balls $\{B(x_i, r_i)\}_{i}$ such that $\textup{Im}(\gamma)\subseteq \bigcup_{i=1}^N B(x_i, r_i)$, where $x_i = \gamma(t_i)$ for some $t_i$ and $r_i = r(t_i)$. We shorten the notation on the geodesics $\eta^\pm_{t_i}\colon [0,r_i]\to \X$ to $\eta^\pm_i$. Up to passing to a subcollection and relabeling, we may suppose that the properties of Lemma \ref{lemma:covering_geodesics_with_balls} are satisfied. In particular, $t_i < t_{i+1}$, $r_1 \ge \sfd(x,x_1)$, $r_N \ge \sfd(y,x_N)$ and $r_i + r_{i+1} \ge \sfd(x_i,x_{i+1})$ for every $i\in \{1,\ldots,N-1\}$.

Next, for $i=1$, we define $s_1^- := \sfd(x,x_1) \le r_1$.
Then, for every $i\in \{1,\ldots,N-1\}$, there exist $0<s^+_i\le r_i$ and $0< s^-_{i+1} \le r_{i+1}$ such that $s^+_i + s^-_{i+1} = \sfd(x_i,x_{i+1})$. For $i=N$, we define $s_N^+ := \sfd(x_N,y) \le r_N$.

With this setup, for every $i\in \{1,\ldots,N\}$, we define the restriction $\tilde\eta^\pm_i := \eta_i|_{ [ 0, s^{\pm}_i ] }$ and we denote by $p^\pm_i = \omega( \tilde\eta^\pm_i )$ the corresponding endpoints; recall that the starting points are $x_i = \gamma(t_i)$. Notice that $\tilde\eta_i^\pm(s) \in \mathcal{R}_{\delta}(U)$ for every $s\neq 0$, hence $\ell(\tilde\eta_i^\pm \cap \mathcal{S}_{\delta}(U)) = 0$. 
Moreover,
\begin{equation}
\label{eq:length_eta}
    \begin{aligned}
        \sum_{i=1}^N \left(\ell(\tilde{\eta}^-_i) + \ell(\tilde{\eta}^+_i)\right) &= \sum_{i=1}^N\left(\sfd(p^-_i,x_i) +\sfd(x_i,p_i^+)\right) \\&\stackrel{\eqref{eq:error_eta_regular}}{\le} \sum_{i=1}^N \sfd(\gamma(t_i-s^-_i),\gamma(t_i + s^+_i)) + 2\varepsilon'\sum_{i=1}^N (s_i^- + s^+_i) \\
        &= (1+2\varepsilon') \sum_{i=1}^N (s_i^- + s^+_i)= (1+2\varepsilon')\sfd(x,y),
    \end{aligned}
\end{equation}
where the last equality follows by the construction of $s_i^\pm$.

We need additional geodesic segments to connect the endpoints $p^{\pm}_i$ of the curves $\tilde{\eta}_i^\pm$. To this end, for $i\in \{1,\ldots,N-1\}$, we let $\theta_i$ denote the geodesic joining $p^+_i$ to $p^-_{i+1}$, with a suitably chosen domain of definition. Similarly, we consider a geodesic $\theta_0$ joining $x$ to $p^-_1$ and by $\theta_N$ a geodesic joining $p^+_N$ to $y$. Now, we denote by $\beta^{-}_i$ the composition of the inversion of the curve $\tilde{\eta}_{i}^{-}$ and a suitable translation, and, by a slight abuse of notation, denote by $\tilde{\eta}_i^{+}$ the geodesic obtained by a composition of $\tilde{\eta}_{i}^{+}$ and a translation. The aforementioned translations and domains of $\theta_j$ are chosen so that $\theta_0$ is defined on $[0, \sfd(x,p^{-}_{1})]$ and that the concatenation, up to suitable translations of the domains of definitions, 
$$\gamma_\varepsilon := \theta_0 \star ( \beta_{1}^{-} \star \tilde{\eta}_{1}^{+} )\star \dots \star \theta_{i} \star ( \beta_{i+1}^{-} \star \tilde{\eta}_{i+1}^{+} ) \star \dots \star \theta_{N-1} \star ( \beta_{N}^{-} \star \tilde{\eta}_{N}^{+} ) \star \theta_{N}$$
is well-defined. We claim that $\gamma_\varepsilon$ satisfies the thesis. To this end, we include the curves $\beta_{i}^{-}$, $\tilde{\eta}_{i}^{+}$, and those $\theta_i$ for which $\ell(\theta_i\cap \mathcal{S}_{\delta}(U)) = 0$ in the collection $\mathcal{G}$. The remaining segments $\theta_i$ are included in $\mathcal{B}$. This gives the classification stated in (ii). To see that (iii) holds, we note that the sum of the length of the curves in $\mathcal{B}$ is bounded from above by the following:
\begin{equation}
\label{eq:length_theta}
    \begin{aligned}
        \sum_{i=0}^N \ell(\theta_i) = \sfd(x,p_1^-) + \sum_{i=1}^{N-1} \sfd(p_i^+, p_{i+1}^-) + \sfd(p_N^+,y) &\le 2\varepsilon'\left(s_1^- + \sum_{i=1}^N (s_i^+ + s_{i+1}^-) + s_N^+\right)\\
        &=2\varepsilon'\sfd(x,y) \leq \varepsilon\sfd(x,y).
    \end{aligned}
\end{equation}
Moreover, combining \eqref{eq:length_eta} and \eqref{eq:length_theta}, we deduce that
$$\ell(\gamma_\varepsilon) = \sum_{i=1}^N \ell(\tilde\eta_i^\pm) + \sum_{i=0}^N \ell(\theta_i) \le (1+4\varepsilon')\sfd(x,y) \le (1+\varepsilon)\sfd(x,y).$$
Finally, \eqref{eq:error_eta_regular} implies that $\gamma_\varepsilon$ lies in the $\varepsilon\sfd(x,y)$-neighbourhood of $\gamma$.
This shows that (i) is satisfied. Since the elements of $\mathcal{B}$ and $\mathcal{G}$ are geodesics, the proof is complete.
\end{proof}

\begin{proof}[Proof of Theorem \ref{theorem-avoidance-irregular-part}] 
    Let $x,y\in U$, let $\gamma_0 \colon [0,1] \to U$ be the geodesic segment joining $x$ to $y$, and let $0<\delta\le \delta_0$ and $0<\varepsilon < \sfd(\gamma_0, X \setminus U)$. We define $\xi :=\varepsilon/(2+\varepsilon)$, so that $\frac{1+\xi}{1-\xi} = \varepsilon$. We claim that there exists a recursively defined sequence of curves $( \gamma_i \colon [0,1] \to U )_{ i }$ with the following properties:
    \begin{itemize}
        \item[(i)] $\alpha(\gamma_i) = x$, $\omega(\gamma_i) = y$, $\ell(\gamma_i) \le (\sum_{h=0}^{i-1}\xi^h)(1+\xi)\sfd(x,y)$ and $\gamma_i$ is contained in the $(\sum_{h=0}^{i-1}\xi^h)(1+\xi)\sfd(x,y)$-neighbourhood of $\gamma_0$; 
        \item[(ii)] for every $i\ge 1$, $\gamma_i$ is a concatenation of a finite number of length-minimizing curves $\{\gamma_i^j\}_{j\in J_i}$ of two types:
        \begin{equation*}
            \mathcal{G}_i:=\{j \in J_i:\, \ell(\gamma_i^j \cap \mathcal{S}_\delta(U))=0\}\quad \text{and}\quad \mathcal{B}_i:=\{ j \in J_i:\, \ell(\gamma_i^j \cap \mathcal{S}_\delta(U))>0\};
        \end{equation*}
        \item[(iii)] the latter families satisfy $\sum_{j\in \mathcal{B}_i} \ell(\gamma_i^j) \le \xi^i \sfd(x,y)$;
        \item[(iv)] the curves from the former families are nested: $\{\gamma^j_i : j \in \mathcal{G}_i\} \supseteq \{\gamma^j_{i-1} : j \in \mathcal{G}_{i-1}\}$.
    \end{itemize}
    We define $\gamma_1$ as the curve provided by Proposition \ref{prop:recursive_construction} applied to $\gamma_0$, with parameter $\xi$. Then $\gamma_1$ satisfies properties (i), (ii), (iii) and vacuously (iv). Suppose we have constructed $\gamma_i$. The new curve $\gamma_{i+1}$ is obtained by replacing each geodesic segment in $\mathcal{B}_i$ by a curve satisfying the conclusion of Proposition \ref{prop:recursive_construction} with parameter $\xi$. This is a well defined curve since the new curves in Proposition \ref{prop:recursive_construction} have the same endpoints as the original segment. The new curve $\gamma_{i+1}$ satisfies $\alpha(\gamma_{i+1})=x$, $\omega(\gamma_{i+1})=y$, and is contained in $U$, since $U$ is convex. 
    Moreover, by construction, it is a concatenation of a finite number of geodesic segments, so $\mathcal{G}_{i+1}$ and $\mathcal{B}_{i+1}$ are well-defined. On the one hand, condition (iv) is satisfied because the procedure does not modify any segments in $\mathcal{G}_i$. On the other hand, the segments in $\mathcal{B}_{i+1}$ are necessarily obtained by the replacement procedure applied to a geodesic segment in $\mathcal{B}_i$, so 
    $$\sum_{j\in \mathcal{B}_{i+1}} \ell(\gamma_{i+1}^j) \le \xi\sum_{j\in \mathcal{B}_i} \ell(\gamma_{i}^j) \le \xi^{i+1} \sfd(x,y),$$
    by property (iii) of $\gamma_i$ and Proposition \ref{prop:recursive_construction}.(iii). Since the elements of $\mathcal{G}_i$ are unchanged, properties (i) and (iii) of $\gamma_i$ and Proposition \ref{prop:recursive_construction}.(i) imply the estimate
    \begin{equation*}
    \begin{aligned}
    \ell(\gamma_{i+1}) \le \ell(\gamma_i) + (1+\xi)\xi^i\sfd(x,y) &\le (1+\xi)\left(\sum_{h=0}^{i-1} \xi^h\right)\sfd(x,y) + (1+\xi)\xi^i\sfd(x,y)\\
    &= (1+\xi)\left(\sum_{h=0}^{i} \xi^h\right)\sfd(x,y).
    \end{aligned}
    \end{equation*}
    Similarly, the triangle inequality and Proposition \ref{prop:recursive_construction}.(i) imply that $\gamma_{i+1}$ is contained in the $(1+\xi)\left(\sum_{h=0}^{i} \xi^h\right)\sfd(x,y)$-neighbourhood of $\gamma_0$. This is property (i) for $\gamma_{i+1}$.

    We have $\ell(\gamma_i) \le \frac{1+\xi}{1-\xi}\sfd(x,y)$ and $\alpha(\gamma_i) = x$ for every $i\in \N$. Therefore, we can apply Arzela--Ascoli's Theorem to the constant-speed reparametrization of $\gamma_i$ in order to find a limit curve $\gamma_\infty$ satisfying $\alpha(\gamma_\infty) = x$, $\omega(\gamma_\infty) = y$, $\ell(\gamma_\infty) \le \frac{1+\xi}{1-\xi}\sfd(x,y) = \varepsilon \sfd(x,y)$. Moreover, $\gamma_\infty$ is contained in the $\frac{1+\xi}{1-\xi}\sfd(x,y)$-neighbourhood of $\gamma_0$, hence in the $\varepsilon$-neighbourhood of $\gamma_0$. The latter neighbourhood is contained in $U$ by the choice of $\varepsilon$.
    We claim that the limit curve satisfies $\ell(\gamma_\infty \cap \mathcal{S}_{\delta}(U)) = 0$. To this end, fix an error $\lambda > 0$ and find $i_0 \in \N$ such that $\sum_{j \in \mathcal{B}_{i_0}} \ell(\gamma_{i_0}^j) < \lambda$ (it suffices to take $i_0 > \frac{\lambda}{\sfd(x,y)\log(1/\xi)}$). Because of (iv), the curve $\gamma_\infty$ contains all the segments in $\mathcal{G}_{i_0}$. Moreover, the remaining part has length at most $\frac{1+\xi}{1-\xi} \lambda$. Therefore, it holds that
    $$\ell(\gamma_\infty \cap \mathcal{S}_{\delta}(U)) \le \sum_{j\in \mathcal{G}_{i_0}}\ell(\gamma_{i_0}^j \cap \mathcal{S}_{\delta}(U)) + \frac{1+\xi}{1-\xi}\lambda = \frac{1+\xi}{1-\xi}\lambda,$$
    because of the definitions of $\mathcal{G}_{i_0}$ and $\mathcal{B}_{i_0}$. By the arbitrariness of $\lambda$, we deduce that $\ell(\gamma_\infty \cap \mathcal{S}_{\delta}(U)) = 0$.
    This shows the thesis.
\end{proof}

\subsection{Avoiding negligible sets in the regular part}
    In this section, we prove Proposition \ref{proposition-regular-part}.
\begin{proof}[Proof of Proposition \ref{proposition-regular-part}]
    Let $\varepsilon > 0$. By applying Proposition \ref{prop:properties_tiny_balls}.(iii), we find the function $\eta \colon ( 0, \delta_0 ) \to (0,\infty)$ and small enough $\delta = \delta(\varepsilon) \le \delta_0$ such that $\eta(\delta) \le (1+\varepsilon)^{1/3} - 1$. We claim that this $\delta$ satisfies the thesis.
    
    Let $x,y \in U$ such that the geodesic $\gamma\colon [0,\sfd(x,y)] \to U$ joining them is contained in $\mathcal{R}_{\delta}(U)$. Let $k\in \N$ be the integer provided by Proposition \ref{prop:properties_tiny_balls}.(iii) applied to the connected set $\textup{Im}(\gamma)$. Hence, for every $t\in [0,\sfd(x,y)]$, there exists $r(t) > 0$ and a $(1+\varepsilon)^{1/3}$-biLipschitz and open map $F_t\colon B(\gamma(t),r(t)) \to \R^k$.

    By compactness, we find $M$ balls $\{B(x_i, r_i)\}_{i}$ such that $\textup{Im}(\gamma)\subseteq \bigcup_{i=1}^M B(x_i, r_i)$, where $x_i = \gamma(t_i)$ for some $t_i$ and $r_i = r(t_i)$. The corresponding map $F_{x_i}$ will be denoted by $F_i$. Up to passing to a subcollection and relabeling, we can suppose that the properties of Lemma \ref{lemma:covering_geodesics_with_balls} are satisfied. In particular, $t_i < t_{i+1}$, $r_1 \ge \sfd(x,x_1)$, $r_M \ge \sfd(y,x_M)$ and $r_i + r_{i+1} \ge \sfd(x_i,x_{i+1})$ for every $i\in \{1,\ldots,M-1\}$. For every $i\in \{1,\ldots,M-1\}$, we consider points $y_i \in B(x_i,r_i)\cap B(x_{i+1},r_{i+1}) \cap \textup{Im}(\gamma)$. To simplify the notation, set $y_0 := x$ and $y_M := y$. Notice that for every $i\in \{1,\ldots,M\}$, the points $y_{i-1}$ and $y_{i}$ belong to the same ball $B(x_i,r_i)$. For every $i\in \{1,\ldots,M\}$, we define the interval $I_i:=[\gamma^{-1}(y_{i-1}),\gamma^{-1}(y_{i})]=[a_i,b_i]$ and let $\gamma_i\colon I_i \to U$ be the restriction of $\gamma$ to $I_i$; the image is contained in $B(x_i,r_i)$. Define $\theta_i := F_i\circ \gamma_i \colon I_i \to \R^k$. For every $v\in \R^k$, we consider the curve 
    $$\theta_i^v(s) := \theta_i(s) + v\sin\left(\frac{\pi}{\vert I_i\vert}(s-a_i)\right), \quad s\in I_i.$$
    We notice that $\alpha(\theta_i^v) = F_i(y_{i-1})$ and $\omega(\theta_i^v) = F_i(y_i)$ for every $v\in \R^k$. Since $F_i(B(x_i,r_i)) =: W_i \subseteq \R^k$ is open and $\text{Im}(\theta_i)$ is compact in $W_i$, there exists a neighbourhood $\Omega_i \subseteq \R^k$ containing $0$ such that $\theta_i^v(I_i) \subseteq W_i$ for every $v \in \Omega_i$. 

    Let $N\subseteq U$ be $\mu$-negligible. In particular, $\mathcal{H}^k(N\cap B(x_i,r_i)) = 0$. Therefore, $E_i := F_i(N \cap B(x_i,r_i))\subseteq W_i$ is negligible for the Lebesgue measure.
    Consider the set $E'_i := \{ (s,v) \in I_i \times \Omega_i \,:\, \theta^v_i(s) \in E_i \}$. For $s \in I_i$, the affine map $v \mapsto L_s(v) := \theta^v_i(s)$ is invertible a.e., so $L_{s}^{-1}(E_i) = \{ v \in \R^k \,:\, \theta^v_i(s) \in E_i \}$ is negligible for the $k$-dimensional Lebesgue measure. Thus Fubini's theorem yields the negligibility of $E'_i$ for the $(k+1)$-dimensional Lebesgue measure. Then, again by Fubini's theorem, for almost every $v \in \Omega_i$, the set $E^v_i := \{ s \in I_i \,:\, \theta_i^v(s) \in E_i \}$ is negligible for the $1$-dimensional Lebesgue measure. For these $v$'s, $\ell( \theta_i^v \cap E_i ) = 0$, and so also $\gamma^v_i := F^{-1} \circ \theta^v_i$ satisfies $\ell( \gamma^v_i \cap N ) = 0$. Let $v_i^j \in \Omega_i$ be a sequence converging to zero and satisfying $\ell(\gamma_i^{v^j_i} \cap N) = 0$. The Euclidean lengths of $\theta^{v^j_i}_i$ converge to the length of $\theta_i$ so 
    $$\ell( \gamma_i^{v^j_i} ) \leq (1+\eta)\ell(\theta_i^{v_i^j}) \le (1+\eta)^2\ell(\theta_i) \le (1+\eta)^3 \ell( \gamma_i )$$ 
    for large enough $j$. Fix such a $j$ and call the curve $\gamma_i^{v_i^j} =: \gamma_i'$. 

    We now consider the curve obtained as concatenation of the $\gamma_i'$'s, namely
    $$\gamma' := \gamma_1'\star \ldots \star \gamma_{N-1}'.$$
    By construction, $\alpha(\gamma') = x$, $\omega(\gamma') = y$, $\ell(\gamma' \cap N) = 0$ and we conclude by computing
    $$\ell(\gamma') \le \sum_{i=1}^{N-1} \ell(\gamma_i') \le (1+\eta)^3\ell(\gamma) \le (1+\varepsilon) \ell(\gamma) = (1+\varepsilon) \sfd(x,y).$$
\end{proof}

\bibliographystyle{abbrv}
\bibliography{biblio.bib}

\end{document}